\documentclass[12pt]{amsart}
\usepackage[margin=1in]{geometry}
\usepackage[dvipsnames,svgnames,x11names]{xcolor}
\usepackage[hyphens]{url}
\usepackage[colorlinks=true,
            linkcolor=Maroon,
            citecolor=blue,
            urlcolor=blue,
            hypertexnames=false,
            linktocpage]{hyperref}
\usepackage{bookmark}
\usepackage{amsmath,thmtools,mathtools,amssymb}
\usepackage{tikz}
\usepackage{float}
\usetikzlibrary{arrows.meta,calc}
\definecolor{FigureBlue}{HTML}{1F6282}
\definecolor{FigureRed}{HTML}{C1533B}
\definecolor{FigureGreen}{HTML}{347252}
\definecolor{BodySlate}{HTML}{294F60}
\definecolor{SectionPine}{HTML}{486150}
\definecolor{VectorBrick}{HTML}{985044}
\definecolor{FigureInk}{HTML}{25292D}
\definecolor{FigureGraphite}{HTML}{62686D}
\definecolor{FigureMidGray}{HTML}{8A8F93}
\definecolor{FigurePaleFill}{HTML}{F3F4F4}
\mathtoolsset{showonlyrefs=true}
\allowdisplaybreaks

\newcommand{\R}{\mathbb{R}}
\newcommand{\Sn}{S^{n-1}}
\newcommand{\ip}[2]{\langle #1,#2\rangle}
\newcommand{\abs}[1]{\left|#1\right|}
\newcommand{\eq}[1]{\begin{equation}\begin{alignedat}{2}#1\end{alignedat}\end{equation}}

\theoremstyle{plain}
\newtheorem{theorem}{Theorem}[section]

\newtheorem{lemma}[theorem]{Lemma}
\newtheorem{corollary}[theorem]{Corollary}

\theoremstyle{definition}

\theoremstyle{remark}

\numberwithin{equation}{section}

\title[Centro-affine Poincar\'e inequality]{Centro-affine Poincar\'e inequality:\\ Unconditional convex bodies}
\author[Y. Hu, M. N. Ivaki]{Yingxiang Hu, Mohammad N. Ivaki}

\begin{document}

\begin{abstract}
We prove that the centro-affine Poincar\'e inequality holds with constant $n$ for every $C^{\infty}_+$ unconditional convex body and every smooth function with natural orthogonality conditions.
\end{abstract}

\maketitle

\section{Introduction}

Throughout the paper, $n\geq 2$. We use the notation and conventions
introduced in \cite{HI26}. A set is unconditional if it is invariant under every coordinate reflection. Let $K\subset\R^n$ be a $C^{\infty}_+$ unconditional convex body, $h=h_K$ be its support function,
\eq{
X=\bar{\nabla}h+hx
}
be its inverse Gauss map, and let
\eq{
\tau=\bar{\nabla}^2h+h\bar{g},\quad g=\frac{1}{h}\tau,\quad dV_K=h\det(\bar{g}^{ik}\tau_{kj})\,dx=\frac{h}{\mathcal{K}}\,dx.
}
Here $\bar{g}$ and $\bar{\nabla}$ denote the standard round metric and its Levi-Civita connection on $\Sn$. The tensor $g$ is the centro-affine metric, and $dV_K$ is the rescaled cone-volume measure. We denote the centro-affine connection by $\nabla$, and the centro-affine Laplacian by $\Delta=\Delta_K$.

\begin{theorem}\label{thm:main-poincare}
Let $K\subset\R^n$ be a $C^{\infty}_+$ unconditional convex body and $F\in C^{\infty}(\Sn)$ satisfy
\eq{
\int_{\Sn}F\, dV_K=0,\quad \int_{\Sn}F\frac{x_i}{h}\,dV_K=0,\quad i=1,\ldots,n.
}
Then
\eq{\label{eq:main-even-poincare}
n\int_{\Sn}F^2\,dV_K\le \int_{\Sn}\abs{\nabla F}_g^2\,dV_K.
}
Equality holds if and only if $F\equiv 0$.
\end{theorem}

For unconditional functions, \autoref{thm:main-poincare} was proved in \cite[Thm. 8.12]{HI26} using a flat logarithmic centro-affine geometry on $(0,\infty)^n$. An earlier proof of the inequality, without the characterization of the equality cases, was obtained in \cite{KM22}. \autoref{thm:main-poincare} extends the recent work of \cite{Iff26} from the class of $O(n-1)\times O(1)$-invariant origin-symmetric convex bodies to all unconditional convex bodies. See also \cite{IM23a,Han23}.

As a consequence of the local-to-global principle \cite{BCD17,CHLL20,Put21,Mil25}, we obtain the even log-Minkowski inequality when one of the bodies is unconditional.

\begin{corollary}\label{cor:unconditional-log-minkowski}
Let $K\subset\R^n$ be a $C^{\infty}_+$ unconditional convex body.
Then, for every origin-symmetric convex body $L\subset\R^n$,
\eq{
 \frac{1}{\operatorname{Vol}(K)}
 \int_{\Sn}\log\left(\frac{h_L}{h_K}\right)dV_K
 \geq
\log\left(
 \frac{\operatorname{Vol}(L)}{\operatorname{Vol}(K)}
 \right).
}
Equality holds if and only if $L=cK$ for some $c>0$.
\end{corollary}

\begin{proof}
Let $\lambda_{1,e}(-\Delta_K)$ denote the first nonzero eigenvalue of $-\Delta_K$ on even functions. By standard spectral theory and elliptic regularity, $\lambda_{1,e}(-\Delta_K)$ is attained by a smooth even eigenfunction. In view of \autoref{thm:main-poincare}, $\lambda_{1,e}(-\Delta_K)>n$.

Let $\mathcal{F}$ be the family of $C^{\infty}_+$ unconditional convex bodies. The convex bodies whose support functions are $ h_t=(1-t)+th_K$, $t\in[0,1]$, remain in $\mathcal{F}$. Now applying \cite[Thm. 2.6]{IM23a} to $\mathcal{F}$ with $p=0$ proves the log-Minkowski inequality for $K$, including the characterization of equality.
\end{proof}

We also obtain the following uniqueness result in the supercritical range $p\in [-n-1,-n)$ when $n\ge 3$.

\begin{corollary}\label{cor:uniqueness}
Let $n\ge 3$ and $-n-1\leq p<-n$. Suppose $K\subset\R^n$ is a $C^{\infty}_+$ unconditional convex body such that $\frac{h^{1-p}}{\mathcal{K}}=1$. Then $K$ is the unit ball.
\end{corollary}
For $n=2$, Andrews \cite{And03} established the uniqueness without any symmetry assumption for $p\in[-7,-2)$. Moreover, in \cite{Du25}, Du proves the uniqueness for strongly symmetric convex bodies in the supercritical range $p\in(-2n-3,-n)$.

Two additional ingredients are needed to pass from unconditional functions to the class in \autoref{thm:main-poincare}. We prove the inequality first for functions that are odd in one coordinate variable and even in all the others, and then for functions that are odd in at least two coordinates variables. The first case follows from our inequality in \cite{HI26}; see \autoref{thm:log-arbitrary}. Before treating the second case, we discuss a few attempts based on factorization.

Let $E_1,\ldots,E_n$ denote the standard coordinate basis of $\R^n$.  Put
\eq{
 l_i=\frac{x_i}{h},\quad P_{ij}=l_iX_j+l_jX_i,\quad i\ne j.
}
We will repeatedly use the standard identity
\eq{\label{eq:conormal-identity}
 \Delta l_i+(n-1)l_i=0,\quad i=1,\ldots,n.
}
If $F$ is odd in both $x_i$ and $x_j$ variables, then both $F/(l_il_j)$ and $F/P_{ij}$ are smooth. Thus one may write $F=\Psi u$, with $\Psi=l_il_j$ or $\Psi=P_{ij}$, and use the identity
\eq{
\int_{\Sn}\abs{\nabla F}_g^2\,dV_K-n\int_{\Sn}F^2\,dV_K
=\int_{\Sn}\Psi^2\abs{\nabla u}_g^2\,dV_K-\int_{\Sn}\Psi(\Delta\Psi+n\Psi)u^2\,dV_K.
}
Consequently, the desired inequality for $F$ would follow from
\eq{
 \Psi(\Delta\Psi+n\Psi)\le 0.
}

This direct argument is useful under additional assumptions. For example, if
\eq{
 h\bar{\nabla}\log\frac{h^{n+1}}{\mathcal{K}}\le nX
}
holds on the positive orthant, then $P_{ij}$ satisfies
$P_{ij}(\Delta P_{ij}+nP_{ij})\le 0$. 

The function $\Psi=l_1l_2$ also gives a \emph{complete}  proof for every unconditional body in $\R^2$. Indeed, on $S^1\cap\{x_1>0,\,x_2>0\}$,
\eq{
 \Delta(l_1l_2)+2l_1l_2=-\frac{2X_1X_2}{h^3(h''+h)}<0.
}

For $n>2$, however, one would need non-positivity of the right-hand side of
\eq{
(\Delta+n)(l_il_j)
=(2-n)l_il_j+2g(\nabla l_i,\nabla l_j).
}
The right-hand side has no definite sign under unconditionality alone.

Let us also mention a half-space capillary example where $l_il_j$ is again successful. Let $\theta\in(0,\pi/2)$,
\eq{
 \Sn_{\theta}=\{x\in\Sn:x_n\ge\cos\theta\},\quad
 \ell(x)=1-\cos\theta\,x_n,
}
and $g_{\theta}$, $dV_{\theta}$, and $\Delta_{\theta}$ be the centro-affine metric, cone-volume measure, and Laplacian of the capillary cap $ \Sn_{\theta}-\cos\theta\, E_n$ whose support function is $\ell$. Let $\mu$ be the $\bar{g}$-unit outward normal to $\partial\Sn_{\theta}$ in $\Sn_{\theta}$. For $1\le i<j<n$, $ \psi_{ij}=\frac{x_ix_j}{\ell^2}$ satisfies
\eq{
 \bar{\nabla}_{\mu}\psi_{ij}&=0\quad\text{on }\partial\Sn_{\theta},\\
 \psi_{ij}(\Delta_{\theta}\psi_{ij}+2n\psi_{ij})
 &=-\frac{2\cos\theta\,(x_n-\cos\theta)(x_ix_j)^2}{\ell^5}\le 0.
}
Together with a decomposition of $F$ and the inequality for the unconditional component of the decomposition from \cite[Thm. 1.2]{CHI26}, this yields
\eq{
 2n\int_{\Sn_{\theta}}F^2\,dV_{\theta}
 \le\int_{\Sn_{\theta}}\abs{\nabla F}_{g_{\theta}}^2\,dV_{\theta}
}
for mean-zero functions satisfying the following horizontal-evenness and Neumann conditions:
\eq{
 F(-x_1,\ldots,-x_{n-1},x_n)=F(x_1,\ldots,x_{n-1},x_n),\quad \bar{\nabla}_{\mu}F=0.
}

We omit the details, since we expect the same inequality to hold without the horizontal evenness assumption for all Neumann functions satisfying the natural orthogonality conditions against the capillary functions $l_i=x_i/\ell$:
\eq{
 \int_{\Sn_{\theta}}F\,dV_{\theta}=0,
 \quad
 \int_{\Sn_{\theta}}Fl_i\,dV_{\theta}=0, \quad i=1,\ldots,n-1.
}

Powers of $l_il_j$ have also been considered in \cite{KLR25}. On the set where $l_il_j>0$, the function $\Psi=(l_il_j)^{\alpha}$ satisfies
\eq{\label{eq:intro-power-supersolution}
\frac{\Delta\Psi}{\Psi}+n
&=n-2\alpha(n-1)-\alpha(1-\alpha)\left(\frac{\abs{\nabla l_i}_g^2}{l_i^2}
+\frac{\abs{\nabla l_j}_g^2}{l_j^2}\right)+2\alpha^2\frac{g(\nabla l_i,\nabla l_j)}{l_il_j}.
}
For $n\ge 3$ and $\alpha=n/(2(n-1))$, this leads to the sufficient condition from \cite[Thm. 1.9]{KLR25}, written in terms of the Minkowski functional $\varphi=\|\,\cdot\,\|_K$ as
\eq{
 \frac{\varphi_{ii}}{\varphi_i^2}+\frac{\varphi_{jj}}{\varphi_j^2}
 \ge \frac{2n}{n-2}\frac{\varphi_{ij}}{\varphi_i\varphi_j}.
}
In particular, this applies when the mixed derivatives $\varphi_{ij}$ are non-positive; this includes the $\ell_q$-models for all $q\geq 1$; see \cite[Cor. 1.6]{KLR25}.

There is another tempting approach suggested by the logarithmic centro-affine geometry. On the positive orthant
\eq{
 \Omega_+:=\{x\in\Sn:x_i>0,\ i=1,\ldots,n\}
}
let $p_i=\frac{x_iX_i}{h}$, and let $g_{\log}$ and $\mathcal{L}_0^{\log}$ denote the logarithmic centro-affine metric and the weighted Laplacian defined in \cite[Sec. 8]{HI26}. We have
\eq{
\mathcal{L}_0^{\log}p_i+np_i=1.
}
Unconditionality yields $X_k=x_kB_k$ with $B_k$ smooth and positive. Hence
\eq{
 \Psi_{ij}=\sqrt{p_ip_j}=\frac{x_ix_j\sqrt{B_iB_j}}{h}
}
is a natural candidate for applying the factorization procedure. Indeed, we have
\eq{
 \frac{\mathcal{L}_0^{\log}\Psi_{ij}}{\Psi_{ij}}+n
 &=-\frac{1}{4}\abs{\nabla^{\log}\log(p_i/p_j)}_{g_{\log}}^2
 +\frac{1}{2}\left(\frac{1}{p_i}+\frac{1}{p_j}\right),\\
 \abs{\tilde{\nabla}\log(p_i/p_j)}_{g_{\mathrm{av}}}^2
 &=2\left(\frac{1}{p_i}+\frac{1}{p_j}\right),
}
where $g_{\mathrm{av}}:=\frac{1}{2}\sum_{k=1}^n\frac{dp_k^2}{p_k}$ and $\tilde{\nabla}$ denotes its Levi-Civita connection. If $g_{\mathrm{av}}\geq g_{\log}$ (we always have $g_{\mathrm{av}}\geq 2g$), then
\eq{
 \mathcal{L}_0^{\log}\Psi_{ij}+n\Psi_{ij}\leq 0.
}
For coordinate ellipsoids, $g_{\log}=g_{\mathrm{av}}=2g$, and $(\mathcal{L}_0^{\log}+n)\Psi_{ij}=0$; see \cite[Prop. 8.2]{HI26}.

For a general unconditional convex body in dimension $n\ge 3$, unconditionality alone controls neither the mixed term in \eqref{eq:intro-power-supersolution}, nor the analogous terms involving $P_{ij}$, nor the comparison $g_{\mathrm{av}}\geq g_{\log}$. Consequently, obtaining a global pointwise supersolution appears difficult. Nevertheless, our argument follows the same underlying principle after passing to two-dimensional sections. Under the deformation $h_t=he^{tu}$, the support function $s_{\xi,t}$ of $D_{\xi,t}=K_t\cap(\xi+E)-\xi$, where $E=\operatorname{span}\{E_i,E_j\}$, satisfies
\eq{
\left.\frac{d^2}{dt^2}\right|_{t=0}
\log s_{\xi,t}(a)\leq 0,
\quad a\in S^1.
}
Moreover, the inherited reflection symmetries ensure that the initial speed $\dot{s}_{\xi,0}/s_{\xi,0}$ is odd with respect to both coordinate reflections. \autoref{lem:planar-two-reflections}, together with the above pointwise inequality, then yields concavity of the sectional area at $t=0$. Now integration over the sections yields
\eq{
\left.\frac{d^2}{dt^2}\right|_{t=0}
\operatorname{Vol}(K_t)\leq 0.
}
In this sense, the requisite supersolution structure is realized not globally on $\Sn$, but fiberwise through the two-dimensional sections of $K$. The first proof below follows this route. We also provide a shorter second proof relying on the planar log-Brunn--Minkowski inequality applied directly to $D_{\xi,t}$ and $D_{\xi,-t}$.

\section{Across coordinate walls}

For $i\in\{1,\ldots,n\}$, let $R_i$ be the reflection in the coordinate hyperplane $E_i^{\perp}$:
\eq{
 R_i x=x-2\ip{x}{E_i}E_i.
}
We say a function $f$ is even, respectively odd, with respect to $R_i$ if
$f\circ R_i=f$, respectively $f\circ R_i=-f$.

\begin{lemma}\label{lem:double-odd-division}
Let $n\geq2$ and $f\in C^{\infty}(S^{n-1})$ be odd with respect to $R_i$ and $R_j$ with $i\ne j$. Then $f(x)=x_ix_jq(x)$ for some $q\in C^{\infty}(S^{n-1})$. In particular, near $S^{n-1}\cap\{x_i=x_j=0\}$,
\eq{
 f=O(x_i^2+x_j^2),\quad
 \bar{\nabla} f=O\big((x_i^2+x_j^2)^{1/2}\big).
}
\end{lemma}

\begin{proof}
We extend $f$ to $\R^n\setminus\{0\}$ as a one-homogeneous function. Note that a smooth function $\phi$ that is odd in the $x_k$-variable, locally near $\{x_k=0\}$, can be written as
\eq{
 \phi(x)=x_k\int_0^1
 \partial_k\phi(x_1,\ldots,tx_k,\ldots,x_n)\,dt.
}
Hence, the quotient $\phi/x_k$ is smooth and has the same parity as $\phi$ in every other variable. Applying this first with $k=i$ and then with $k=j$, we obtain $f=x_ix_jq$ on $S^{n-1}$, with some smooth $q$, and the estimates follow immediately.
\end{proof}

\begin{lemma}\label{lem:planar-two-reflections}
Let $K\subset\R^2$ be a $C^{\infty}_+$ unconditional convex body and $v\in C^{\infty}(S^1)$ be odd with respect to both coordinate reflections. Then
\eq{
 2\int_{S^1}v^2\,dV_K\leq \int_{S^1}\abs{\nabla v}_{g}^2\,dV_K.
}
The inequality is strict unless $v\equiv 0$.
\end{lemma}

\begin{proof}
Let $x(\theta)=(\cos\theta,\sin\theta)$ for $\theta\in(0,2\pi)$. We have
\eq{
X(\theta)=\bar{\nabla} h+hx=(h\cos\theta-h'\sin\theta,\ h\sin\theta+h'\cos\theta).
}

Since $K$ is unconditional, $x_iX_i\geq 0$ for $i=1,2$. Moreover, we have
\eq{
l_1(\theta)=\frac{\cos\theta}{h(\theta)},\quad
l_2(\theta)=\frac{\sin\theta}{h(\theta)},
}
and hence
\eq{
l_1'(\theta)&=-\frac{h\sin\theta+h'\cos\theta}{h^2}=-\frac{X_2}{h^2},\\
l_2'(\theta)&=\frac{h\cos\theta-h'\sin\theta}{h^2}=\frac{X_1}{h^2}.
}
On the other hand, the centro-affine metric of $K$ is
\eq{
g=\frac{h''+h}{h}\,d\theta^2.
}
Therefore, by \cite[Lem. 2.3]{HI26} and \eqref{eq:intro-power-supersolution}, applied with $\Psi=l_1l_2$, we obtain
\eq{
 \Psi(\Delta\Psi+2\Psi)
 =-\frac{2x_1x_2X_1X_2}{h^5(h''+h)}\leq 0,
}
with strict inequality whenever $x_1x_2\neq0$.

By \autoref{lem:double-odd-division}, $v/(x_1x_2)$ extends smoothly across the coordinate axes, and we may write $v=\Psi f$, where $f$ is a smooth function. Since $\Psi$ vanishes at the endpoints of each quadrant, integration by parts yields
\eq{
 \int_{S^1}\left(|\nabla v|_{g}^2-2v^2\right)dV_K
 =\int_{S^1} \Psi^2|\nabla f|_{g}^2\,dV_K
 -\int_{S^1} f^2\Psi (\Delta\Psi+2\Psi)\,dV_K.
}
Both terms on the right are nonnegative, and the second term is positive unless $v\equiv 0$.
\end{proof}

\begin{lemma}\label{lem:area-second-variation}
Let $(D_t)_{|t|<\delta}$ be a smooth family of $C^{\infty}_+$ convex bodies in $\R^2$, with the origin in their interiors and support functions $s_t$, and let $A(t)=\operatorname{Area}(D_t)$. Here dots denote differentiation with respect to $t$. Put $v=\dot{s}_0/s_0$. Then
\eq{
 \ddot{A}(0)
 &=2\int_{S^1}v^2\,dV_{D_0}
 -\int_{S^1}|\nabla v|_{g_{D_0}}^2\,dV_{D_0}+\int_{S^1}\left(\ddot{s}_0
 -\frac{\dot{s}_0^2}{s_0}\right)dS_{D_0}.
}
If $D_0$ is unconditional, $v$ is odd with respect to both coordinate reflections and $\ddot{s}_0\leq \frac{\dot{s}_0^2}{s_0}$, then
\eq{
 \ddot{A}(0)\leq 0.
 }
\end{lemma}
\begin{proof}
The derivation of $\ddot{A}$ is standard. Let $\dot{s}_t=f_ts_t$. Then
\eq{
\dot{A}(t)&=\int_{S^1}f_t\,dV_{D_t},\\
\ddot{A}(t)&=\int_{S^1}\dot{f}_t\,dV_{D_t}
+\int_{S^1}f_t(\Delta_{D_t}f_t+2f_t)\,dV_{D_t}.
}
Since $\dot{f}_t=\frac{\ddot{s}_t}{s_t}-\frac{\dot{s}_t^2}{s_t^2}$ and $f_0=v$, the claim follows from integration by parts. The second statement follows from \autoref{lem:planar-two-reflections}.
\end{proof}

\begin{lemma}\label{lem:two-coordinates-reflections}
Let $n\geq 2$ and suppose $K\subset\R^n$ is a $C^{\infty}_+$ unconditional convex body. Assume $u\in C^{\infty}(\Sn)$ is odd with respect to $R_i$ and $R_j$ for some $1\le i<j\le n$. Then
\eq{
n\int_{\Sn}u^2\,dV_K\le \int_{\Sn}\abs{\nabla u}_g^2\,dV_K.
}
The inequality is strict unless $u\equiv 0$.
\end{lemma}

\begin{proof}[First proof of \autoref{lem:two-coordinates-reflections}]
For $n=2$, this is \autoref{lem:planar-two-reflections}, so we assume $n\geq 3$. As usual, we extend support functions from the unit sphere as homogeneous functions of degree one and functions such as $u$ as homogeneous functions of degree zero on $\R^n\setminus\{0\}$.

Let $E=\operatorname{span}\{E_i,E_j\}$ and $F=E^{\perp}$. For $\xi\in F$, set
\eq{
 A_{\xi}(t):=\operatorname{Area}\big(K_t\cap(\xi+E)\big).
}
We will show that under the deformation $h_t=he^{tu}$ of $K$, the support function $s_{\xi,t}$ of $(K_t\cap(\xi+E))-\xi$ satisfies
\eq{
 \ddot{s}_{\xi,0}\leq \frac{\dot{s}_{\xi,0}^2}{s_{\xi,0}}
 \quad\text{on }S^1=E\cap\Sn.
}
The symmetries of the section, together with \autoref{lem:planar-two-reflections} and \autoref{lem:area-second-variation}, imply that
\eq{
 \ddot{A}_{\xi}(0)\leq 0.
}
On the other hand, by Fubini's theorem
\eq{
 \operatorname{Vol}(K_t)=\int_{F} A_{\xi}(t)\,d\xi,
}
while
\eq{\label{eq:volume-second-variation}
& \left.\frac{d^2}{dt^2}\right|_{t=0}\operatorname{Vol}(K_t)
 &=\int_{\Sn}u(\Delta u+nu)\,dV_K=\int_{\Sn}\left(nu^2-\abs{\nabla u}_g^2\right)\,dV_K.
}
Hence, once differentiation under the integral sign is justified, the concavity of the sectional area functional yields that of $\operatorname{Vol}(K_t)$ and thus the desired inequality.

For sufficiently small $|t|$, $ h_t=he^{tu}$ is the support function of a $C^{\infty}_+$ convex body $K_t$, and
\eq{\label{eq:global-log-symmetries}
 R_iK_t=R_jK_t=K_{-t},\quad R_iR_jK_t=K_t.
}
Moreover, $u=0$ on $F\setminus\{0\}$. Hence $h_t|_{F}=h|_{F}$, and
\eq{
 C:=\mathcal{P}_{F}K_t=\mathcal{P}_{F}K
}
is independent of $t$. Here $\mathcal{P}_{F}$ denotes the orthogonal projection onto $F$.

Let $\xi\in\operatorname{int}(C)$ (relative interior of $C$ in $F$). Consider the two-dimensional sections
\eq{
 D_{\xi,t}=\big(K_t\cap(\xi+E)\big)-\xi\subset E.
}
Let $s_{\xi,t}$ denote the support function of $D_{\xi,t}$. Since $\xi\in\operatorname{int}(C)$, $\partial K_t$ and $\xi+E$ intersect transversely; hence $D_{\xi,t}$ is a $C^{\infty}_+$ convex body in $E$.

In view of \eqref{eq:global-log-symmetries},
\eq{\label{eq:section-symmetries}
 R_iD_{\xi,t}=R_jD_{\xi,t}=D_{\xi,-t},\quad -D_{\xi,t}=D_{\xi,t}.
}
In particular, $D_{\xi,0}$ is unconditional and contains the origin in its interior. By \eqref{eq:section-symmetries},
\eq{
 s_{\xi,-t}(a)=s_{\xi,t}(R_i a)=s_{\xi,t}(R_j a),\quad a\in E.
}
Consequently, $v_{\xi}:=\dot{s}_{\xi,0}/s_{\xi,0}$ is odd with respect to both coordinate reflections of $E$.

Note that for $a\in E\setminus\{0\}$:
\eq{\label{eq:section-support}
 s_{\xi,t}(a)=\min_{b\in F}\left\{h_t(a+b)+h_{\xi+E}(-b)\right\}=\min_{b\in F}\left\{h_t(a+b)-\ip{b}{\xi}\right\}.
}
Indeed, this follows from \cite[Thm. 5.6]{Roc70}; see also \cite[Sec. 1.7, Notes, item 3, p. 59]{Sch14} for intersections of convex bodies. We nevertheless include the proof, since we will use below that the minimizer $b$ is unique (due to smoothness and strict convexity) and satisfies
\eq{
 \mathcal{P}_{F}Dh_t(a+b)=\xi.
}

Let $y\in D_{\xi,t}$ and $b\in F$. Then $y+\xi\in K_t$ and
\eq{
 \ip{a}{y}=\ip{a+b}{y+\xi}-\ip{b}{\xi}\leq h_t(a+b)-\ip{b}{\xi}.
}
Taking the maximum over $y$ and then the infimum over $b$, we find
\eq{
 s_{\xi,t}(a)\leq
 \inf_{b\in F}\left\{h_t(a+b)-\ip{b}{\xi}\right\}.
}

For the reverse inequality, let $G_{\xi,t}:\partial D_{\xi,t}\to S^1$ be the Gauss map of the section and $\nu_t:\partial K_t\to\Sn$ be the Gauss map of $K_t$. Put
\eq{
 y=G_{\xi,t}^{-1}\left(\frac{a}{|a|}\right),\quad  b=\frac{|a|}{|\mathcal{P}_{E}\nu_t(y+\xi)|}\mathcal{P}_{F}\nu_t(y+\xi).
}
By the definition of the Gauss map of the section,
\eq{
 \frac{\mathcal{P}_{E}\nu_t(y+\xi)}{|\mathcal{P}_{E}\nu_t(y+\xi)|}=\frac{a}{|a|}.
}
Consequently,
\eq{
 a+b=\frac{|a|}{|\mathcal{P}_{E}\nu_t(y+\xi)|}\nu_t(y+\xi),
 \quad Dh_t(a+b)=y+\xi,
}
and
\eq{
 h_t(a+b)-\ip{b}{\xi}=\ip{a+b}{y+\xi}-\ip{b}{\xi}=\ip{a}{y}=s_{\xi,t}(a).
}
This proves \eqref{eq:section-support}. Also note that any minimizer $\tilde{b}$ must satisfy $h_t(a+\tilde{b})=\ip{a+\tilde{b}}{y+\xi}$. In particular, both $\frac{a+b}{|a+b|}$ and $\frac{a+\tilde{b}}{|a+\tilde{b}|}$ are outer unit normals to $K_t$ at $y+\xi$. Hence $b=\tilde{b}$. Moreover, from $Dh_t(a+b)=y+\xi$ it follows that $\mathcal{P}_{F}Dh_t(a+b)=\xi$.

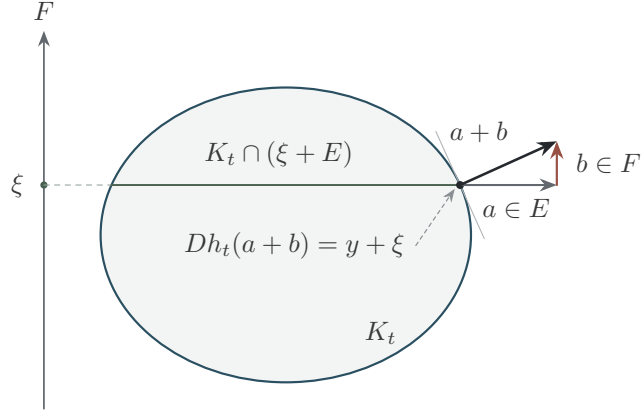
\begin{figure}[H]
\begin{tikzpicture}[
  >=Stealth,
  line cap=round,
  line join=round,
  every node/.style={font=\footnotesize,text=black!78}
]
  \pgfmathsetmacro{\radiusE}{2.45}
  \pgfmathsetmacro{\radiusF}{1.95}
  \pgfmathsetmacro{\sectionheight}{0.66}
  \pgfmathsetmacro{\contactx}{
    \radiusE*sqrt(1-(\sectionheight/\radiusF)^2)}
  \pgfmathsetmacro{\vectorrun}{1.28}
  \pgfmathsetmacro{\vectorrise}{
    \vectorrun*\sectionheight*\radiusE^2/(\contactx*\radiusF^2)}
  \pgfmathsetmacro{\tangenthalfwidth}{0.32}
  \pgfmathsetmacro{\tangenthalfheight}{
    \tangenthalfwidth*\contactx*\radiusF^2/
    (\sectionheight*\radiusE^2)}
  \pgfmathsetmacro{\negativetangenthalfwidth}{-\tangenthalfwidth}
  \pgfmathsetmacro{\negativetangenthalfheight}{-\tangenthalfheight}

  \coordinate (p) at (\contactx,\sectionheight);
  \coordinate (q) at (-\contactx,\sectionheight);
  \coordinate (pa) at ($(p)+(\vectorrun,0)$);
  \coordinate (pab) at ($(pa)+(0,\vectorrise)$);

  \path[fill=FigurePaleFill,draw=BodySlate,line width=.82pt]
    (0,0) ellipse [x radius=\radiusE cm,y radius=\radiusF cm];

  \draw[->,FigureGraphite,line width=.65pt] (-3.20,-2.30)--(-3.20,2.70)
    node[above] {$F$};
  \fill[SectionPine] (-3.20,\sectionheight) circle (1.35pt);
  \node[anchor=east]
    at (-3.32,\sectionheight) {$\xi$};
  \draw[SectionPine!58,dash pattern=on 2pt off 2pt,line width=.4pt]
    (-3.20,\sectionheight)--(q);
  \draw[SectionPine,line width=.72pt] (q)--(p);
  \node[anchor=south]
    at ($(q)!.48!(p)+(0,.08)$) {$K_t\cap(\xi+E)$};

  \draw[FigureMidGray!65,line width=.42pt]
    ($(p)+(\negativetangenthalfwidth,\tangenthalfheight)$)--
    ($(p)+(\tangenthalfwidth,\negativetangenthalfheight)$);

  \fill[FigureInk] (p) circle (1.5pt);
  \draw[->,FigureGraphite,line width=.78pt] (p)--(pa)
    node[pos=.58,below=1pt] {$a\in E$};
  \draw[->,VectorBrick,line width=.82pt] (pa)--(pab)
    node[midway,right=3pt] {$b\in F$};
  \draw[->,FigureInk,line width=.98pt] (p)--(pab)
    node[pos=.64,above left=1pt] {$a+b$};

  \node[anchor=north east] (contactnote) at (1.72,0.18)
    {$Dh_t(a+b)=y+\xi$};
  \draw[FigureMidGray,->,dash pattern=on 1.5pt off 1.4pt,line width=.45pt]
    (contactnote.east)--($(p)+(-0.07,-0.05)$);
  \node[anchor=west] at (0.9,-1.3) {$K_t$};
\end{tikzpicture}
\end{figure}

We next justify the smooth dependence of the sections. Define
\eq{
 \Phi&:F\times\operatorname{int}(C)\times
 (E\setminus\{0\})\times(-t_0,t_0)\to F,\\
 \Phi(b,\xi,a,t)&=\mathcal{P}_{F}Dh_t(a+b)-\xi.
}
The unique minimizer $b=b(\xi,a,t)$ in \eqref{eq:section-support} satisfies
\eq{
 \Phi(b(\xi,a,t),\xi,a,t)=0.
}
For fixed $(b,\xi,a,t)$, the derivative of $\Phi$ with respect to $b$ is the linear map
\eq{
D_b\Phi(b,\xi,a,t)&:F\to F,\\
D_b\Phi(b,\xi,a,t)[q]&=\mathcal{P}_{F}\big(D^2h_t(a+b)q\big), \quad q\in F.
}
If $0\ne q\in F$, then
\eq{
 \ip{D_b\Phi(b,\xi,a,t)[q]}{q}=D^2h_t(a+b)[q,q]>0,
}
where we used $q\notin \ker(D^2h_t(a+b))=\R (a+b)$. The implicit function theorem and the uniqueness of the minimizer $b$ now show that $b=b(\xi,a,t)$ depends smoothly on $(\xi,a,t)$, and
\eq{
 s_{\xi,t}(a)=h_t\big(a+b(\xi,a,t)\big) -\ip{b(\xi,a,t)}{\xi}
}
depends smoothly on $(\xi,a,t)$ for $\xi\in\operatorname{int}(C)$ and $a\in E\setminus\{0\}$. It follows that $A_{\xi}(t)$ depends smoothly on $(\xi,t)$ for $\xi\in\operatorname{int}(C)$.

Let $b$ be the minimizer at $t=0$, and set $w=a+b$. Then
\eq{
 \mathcal{P}_{F}Dh(w)=\xi,\quad
 s_{\xi,0}(a)=h(w)-\ip{b}{\xi}.
}
For $k\notin\{i,j\}$, unconditionality implies $R_kDh(w)\in K$. Since the hyperplane with normal $w$ supports $K$ at $Dh(w)$,
\eq{
 0\leq\ip{w}{Dh(w)}-\ip{w}{R_kDh(w)}=2w_k(Dh(w))_k.
}
Since $b=\mathcal{P}_{F}w$ and $\xi=\mathcal{P}_{F}Dh(w)$, it follows that
\eq{
 \ip{b}{\xi}=\sum_{k\notin\{i,j\}}w_k(Dh(w))_k\geq 0, \quad 
 h(w)\geq s_{\xi,0}(a).
}

In view of \eqref{eq:section-support},
\eq{\label{eq:for-second-proof}
 s_{\xi,t}(a)\leq h(w)e^{tu(w)}-\ip{b}{\xi},
}
with equality at $t=0$. Consequently,
\eq{\label{eq:section-log-concavity}
 \dot{s}_{\xi,0}(a)&=h(w)u(w),\\
 \ddot{s}_{\xi,0}(a)&\leq h(w)u(w)^2
 \leq\frac{h(w)^2u(w)^2}{s_{\xi,0}(a)}
 =\frac{\dot{s}_{\xi,0}(a)^2}{s_{\xi,0}(a)}.
}

For $f\in C^{\infty}(S^1)$, set
\eq{
 \mathcal{E}_{\xi}(f):=
 \int_{S^1}\left(|\nabla f|_{g_{D_{\xi,0}}}^2-2f^2\right)
 \,dV_{D_{\xi,0}}.
}
By \autoref{lem:planar-two-reflections} and \autoref{lem:area-second-variation}, for $\xi\in\operatorname{int}(C)$, we have
\eq{\label{eq:sectional-energy-pointwise}
 -\ddot{A}_{\xi}(0)\geq\mathcal{E}_{\xi}(v_{\xi})\geq 0.
}

By Fubini's theorem, $ \operatorname{Vol}(K_t)=\int_C A_{\xi}(t)\,d\xi$. We next justify
\eq{\label{eq:volume-concavity}
 \left.\frac{d^2}{dt^2}\right|_{t=0}\operatorname{Vol}(K_t)
 =\int_C \ddot{A}_{\xi}(0)\,d\xi.
}

We first note that, for all small $|t|$,
\eq{
x\in \Sn,\quad \mathcal{P}_{F}Dh_t(x)\in\partial C  \implies x\in F\cap\Sn.
}
Indeed, let $\eta\in F\cap\Sn$ be an outer normal to $C=\mathcal{P}_{F}K_t$ at $\mathcal{P}_{F}Dh_t(x)$. Then
\eq{
 \ip{\eta}{Dh_t(x)}
 =\ip{\eta}{\mathcal{P}_{F}Dh_t(x)}
 =h_C(\eta)=\max_{y\in K_t}\langle\eta,\mathcal{P}_Fy\rangle=\max_{y\in K_t}\langle\eta,y\rangle=h_{K_t}(\eta).
}
Thus $\eta$ is an outer unit normal to $K_t$ at $Dh_t(x)$. By definition, $x$ is the outer unit normal at $Dh_t(x)$. Thus $x=\eta$.

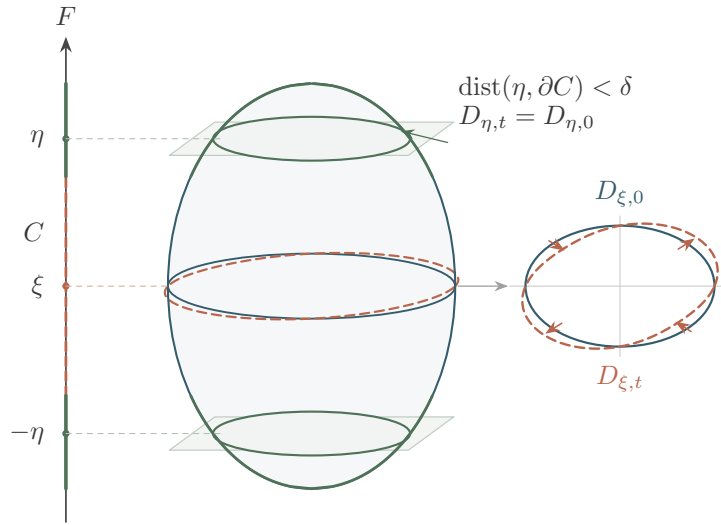
\begin{figure}[H]
\begin{tikzpicture}[
  >=Stealth,
  line cap=round,
  line join=round,
  every node/.style={font=\footnotesize,text=black!78}
]
  \definecolor{FigureBlue}{HTML}{355C6D}
  \definecolor{FigureRed}{HTML}{B9654B}
  \definecolor{FigureGreen}{HTML}{4F7158}

  \pgfmathsetmacro{\bodyradius}{1.90}
  \pgfmathsetmacro{\bodyheight}{2.68}
  \pgfmathsetmacro{\stationaryheight}{1.45}
  \pgfmathsetmacro{\negativestationaryheight}{-\stationaryheight}
  \pgfmathsetmacro{\negativebodyheight}{-\bodyheight}
  \pgfmathsetmacro{\fixedheight}{1.95}
  \pgfmathsetmacro{\negativefixedheight}{-\fixedheight}
  \pgfmathsetmacro{\fixedradius}{
    \bodyradius*sqrt(1-(\fixedheight/\bodyheight)^2)}

  \path[fill=FigureBlue!5]
    (0,0) ellipse [x radius=\bodyradius cm,y radius=\bodyheight cm];

  \path[fill=FigureGreen!7,draw=FigureGreen!36,line width=.45pt]
    (-1.88,1.73)--(1.28,1.73)--(1.88,2.17)--(-1.28,2.17)--cycle;
  \path[fill=FigureGreen!7,draw=FigureGreen!36,line width=.45pt]
    (-1.88,-2.17)--(1.28,-2.17)--(1.88,-1.73)--(-1.28,-1.73)--cycle;
  \draw[FigureGreen,line width=.8pt,fill=FigureGreen!8]
    (0,\fixedheight) ellipse [x radius=\fixedradius cm,y radius=.29cm];
  \draw[FigureGreen,line width=.8pt,fill=FigureGreen!8]
    (0,\negativefixedheight)
    ellipse [x radius=\fixedradius cm,y radius=.29cm];

  \draw[FigureBlue,line width=.72pt]
    (0,0) ellipse [x radius=\bodyradius cm,y radius=.43cm];
  \draw[FigureRed,dash pattern=on 3.5pt off 2.1pt,line width=.82pt]
    plot[domain=0:360,samples=100,variable=\theta]
    ({\bodyradius*(1+.11*sin(2*\theta))*cos(\theta)},
     {.43*(1+.11*sin(2*\theta))*sin(\theta)}) -- cycle;

  \draw[FigureBlue,line width=.78pt]
    (0,0) ellipse [x radius=\bodyradius cm,y radius=\bodyheight cm];

  \draw[FigureGreen,line width=1.05pt]
    plot[domain=\stationaryheight:\bodyheight,samples=45,variable=\z]
    ({\bodyradius*sqrt(1-(\z/\bodyheight)^2)},{\z});
  \draw[FigureGreen,line width=1.05pt]
    plot[domain=\stationaryheight:\bodyheight,samples=45,variable=\z]
    ({-\bodyradius*sqrt(1-(\z/\bodyheight)^2)},{\z});
  \draw[FigureGreen,line width=1.05pt]
    plot[domain=\negativebodyheight:\negativestationaryheight,
      samples=45,variable=\z]
    ({\bodyradius*sqrt(1-(\z/\bodyheight)^2)},{\z});
  \draw[FigureGreen,line width=1.05pt]
    plot[domain=\negativebodyheight:\negativestationaryheight,
      samples=45,variable=\z]
    ({-\bodyradius*sqrt(1-(\z/\bodyheight)^2)},{\z});

  \draw[->,black!72,line width=.65pt] (-3.25,-3.12)--(-3.25,3.30)
    node[above] {$F$};
  \draw[FigureRed,dash pattern=on 3.5pt off 2.1pt,line width=1.05pt]
    (-3.25,-\stationaryheight)--(-3.25,\stationaryheight);
  \draw[FigureGreen,line width=1.35pt]
    (-3.25,\stationaryheight)--(-3.25,\bodyheight);
  \draw[FigureGreen,line width=1.35pt]
    (-3.25,\negativebodyheight)--(-3.25,\negativestationaryheight);
  \fill[FigureGreen] (-3.25,1.95) circle (1.35pt);
  \fill[FigureGreen] (-3.25,-1.95) circle (1.35pt);
  \fill[FigureRed] (-3.25,0) circle (1.35pt);
  \node[anchor=east] at (-3.39,.72) {$C$};
  \node[anchor=east] at (-3.39,1.95) {$\eta$};
  \node[anchor=east] at (-3.39,-1.95) {$-\eta$};
  \node[anchor=east] at (-3.39,0) {$\xi$};
  \draw[FigureGreen!45,dash pattern=on 2pt off 2pt,line width=.4pt]
    (-3.25,1.95)--(-1.15,1.95);
  \draw[FigureGreen!45,dash pattern=on 2pt off 2pt,line width=.4pt]
    (-3.25,-1.95)--(-1.15,-1.95);
  \draw[FigureRed!55,dash pattern=on 2pt off 2pt,line width=.4pt]
    (-3.25,0)--(-1.90,0);

  \node[anchor=west,align=left] (fixednote) at (1.8,2.42)
    {$\operatorname{dist}(\eta,\partial C)<\delta$\\[-1pt]
     $D_{\eta,t}=D_{\eta,0}$};
  \draw[FigureGreen,->,line width=.48pt]
    (fixednote.south west)--(1.18,2.05);

  \draw[black!38,->,line width=.48pt]
    (1.94,0)--(2.60,0);
  \begin{scope}[shift={(4.08,0)}]
    \draw[black!22,line width=.35pt] (-1.46,0)--(1.46,0);
    \draw[black!22,line width=.35pt] (0,-.93)--(0,.93);
    \draw[FigureBlue,line width=.82pt]
      (0,0) ellipse [x radius=1.25cm,y radius=.80cm];
    \draw[FigureRed,dash pattern=on 3.5pt off 2.1pt,line width=.9pt]
      plot[domain=0:360,samples=100,variable=\theta]
      ({1.25*(1+.14*sin(2*\theta))*cos(\theta)},
       {.80*(1+.14*sin(2*\theta))*sin(\theta)}) -- cycle;

    \draw[FigureRed,->,line width=.58pt] (.77,.49)--(.99,.63);
    \draw[FigureRed,->,line width=.58pt] (-.77,-.49)--(-.99,-.63);
    \draw[FigureRed,->,line width=.58pt] (-.94,.60)--(-.73,.47);
    \draw[FigureRed,->,line width=.58pt] (.94,-.60)--(.73,-.47);

    \node[text=FigureBlue,anchor=south] at (0,.91) {$D_{\xi,0}$};
    \node[text=FigureRed,anchor=north] at (0,-.91) {$D_{\xi,t}$};
  \end{scope}
\end{tikzpicture}
\end{figure}

Let us now impose an additional assumption that $u$ vanishes  in an open neighborhood $U\subset \Sn$ of $F\cap\Sn$. Geometrically, the deformation $h_t=he^{tu}$ is stationary for normal directions close to $F$. Consequently, the part of $\partial K$ whose outer normals lie near $F\cap\Sn$ does not move. This is used to keep the planar sections fixed near the boundary of $\mathcal{P}_{F}K$ so that \eqref{eq:volume-concavity} is justified. We now make this precise.

The compact set
\eq{
 \left\{\mathcal{P}_{F}Dh_t(x):|t|\leq t_0,\ x\in\Sn\setminus U\right\}
}
has positive distance from $\partial C $. Choose $\delta>0$ smaller than this distance. Then
\eq{
 \partial K_t\cap
 \left\{y\in\R^n:\operatorname{dist}(\mathcal{P}_{F}y,\partial C)<\delta\right\}
 =\partial K\cap\left\{y\in\R^n:\operatorname{dist}(\mathcal{P}_{F}y,\partial C)<\delta\right\}.
}
Indeed, write $y=Dh_t(x)\in\partial K_t$. If $\operatorname{dist}(\mathcal{P}_Fy,\partial C)<\delta$, the choice of $\delta$ implies that $x\in U$. Since $u=0$ on $U$, we have $Dh_t(x)=Dh(x)$, and hence $y\in\partial K$. The reverse inclusion follows similarly.

We have proved
\eq{
 D_{\xi,t}=D_{\xi,0},\quad A_{\xi}(t)=A_{\xi}(0),\quad \operatorname{dist}(\xi,\partial C)<\delta.
}
Let $C_{\delta}=\{\xi\in C:\,\operatorname{dist}(\xi,\partial C)\geq\delta\}$. Then
\eq{
 \sup_{\xi\in C_{\delta},\ |t|\leq t_0}
 \left|\partial_tA_{\xi}(t)\right|+ \sup_{\xi\in C_{\delta},\ |t|\leq t_0}
 \left|\partial_t^2A_{\xi}(t)\right|<\infty,
}
and
\eq{
 \operatorname{Vol}(K_t)-\operatorname{Vol}(K)
 =\int_{C_{\delta}}\big(A_{\xi}(t)-A_{\xi}(0)\big)\,d\xi.
}
Differentiating this twice and combining \eqref{eq:volume-second-variation}, \eqref{eq:sectional-energy-pointwise}, and \eqref{eq:volume-concavity}, we obtain
\eq{\label{eq:sectional-energy-integrated}
 \int_{\Sn}\left(|\nabla u|_g^2-nu^2\right)dV_K
 =-\int_C \ddot{A}_{\xi}(0)\,d\xi
 \geq\int_{\operatorname{int}(C)}\mathcal{E}_{\xi}(v_{\xi})\,d\xi.
}
This proves the inequality when $u$ vanishes near $F\cap\Sn$. Next we remove this assumption. 

Let $\chi\in C^{\infty}(\R)$ be such that $0\leq\chi\leq1$, $\chi=0$ on $(-\infty,1]$, and $\chi=1$ on $[4,\infty)$. Define
\eq{
 \chi_{\varepsilon}(x) =\chi\left(\frac{|\mathcal{P}_{E}x|^2}{\varepsilon^2}\right),
 \quad u_{\varepsilon}=\chi_{\varepsilon} u,\quad x\in \Sn.
}

In view of \autoref{lem:double-odd-division}, near $F\cap\Sn$, we have
\eq{
 u=O(|\mathcal{P}_{E}x|^2),\quad \bar{\nabla} u=O(|\mathcal{P}_{E}x|).
}
Note that 
\eq{
 \operatorname{supp}(u_{\varepsilon}-u)
 \subset\{x\in\Sn:|\mathcal{P}_Ex|\leq 2\varepsilon\},
}
and hence
\eq{
 \|u_{\varepsilon}-u\|_{C^0(\Sn)}\leq C\varepsilon^2.
}
Moreover, we have
\eq{
 \bar{\nabla}(u_{\varepsilon}-u)
 = (\chi_{\varepsilon}-1)\bar{\nabla}u+u\bar{\nabla}\chi_{\varepsilon}.
}
On the support of $\bar{\nabla}\chi_{\varepsilon}$, we have $\varepsilon\leq|\mathcal{P}_{E}x|\leq2\varepsilon$ and $|\bar{\nabla}\chi_{\varepsilon}|\leq C/\varepsilon$. Hence both terms on the right-hand side are $O(\varepsilon)$, and
\eq{
 \|u_{\varepsilon}-u\|_{C^1(\Sn)}\leq C\varepsilon\to 0.
}
Moreover, $u_{\varepsilon}$ is odd with respect to the same coordinate reflections as $u$, and it vanishes when $|\mathcal{P}_{E}x|\leq\varepsilon$. Therefore, we may apply the previous argument to $u_{\varepsilon}$.

Let $v_{\xi,\varepsilon}$ and $v_{\xi}$ be the section speeds associated with $u_{\varepsilon}$ and $u$, respectively. For $\xi\in\operatorname{int}(C)$ and $a\in S^1$, let $b_{\xi}(a)$ be the minimizer in \eqref{eq:section-support} at $t=0$, and write $w_{\xi}(a)=a+b_{\xi}(a)$. Then
\eq{
 \mathcal{P}_{F}Dh(a+b_{\xi}(a))=\xi.
}
By the discussion above, $w_{\xi}(a)$ depends smoothly on $(\xi,a)$. Note that this map depends only on $K$, not on the perturbation. The first identity in \eqref{eq:section-log-concavity} is local at $t=0$ and does not use the assumption that $u$ vanishes near $F\cap \Sn$. Therefore,
\eq{
 \dot{s}^{(\varepsilon)}_{\xi,0}(a)
 &=h(w_{\xi}(a))u_{\varepsilon}(w_{\xi}(a)),\\
 \dot{s}_{\xi,0}(a)&=h(w_{\xi}(a))u(w_{\xi}(a)).
}
It follows from the smooth dependence on $a$ and
$u_{\varepsilon}\to u$ in $C^1(\Sn)$ that
\eq{
 v_{\xi,\varepsilon}\to v_{\xi}
 \quad\text{in }C^1(S^1).
}
Applying \eqref{eq:sectional-energy-integrated} to $u_{\varepsilon}$, using $u_{\varepsilon}\to u$ in $C^1(\Sn)$, $\mathcal{E}_{\xi}(v_{\xi,\varepsilon})\geq 0$, and Fatou's lemma, we obtain
\eq{\label{eq:sum-sectional-energy}
 \int_{\Sn}\left(|\nabla u|_g^2-nu^2\right)dV_K
 &\geq\int_{\operatorname{int}(C)}
 \mathcal{E}_{\xi}(v_{\xi})\,d\xi\geq 0.
}

Finally, suppose that $u\not\equiv 0$, and choose $x\in\Sn$ with $u(x)\ne 0$. Then $a=\mathcal{P}_{E}x\ne 0$, and the implication proved above at $t=0$ shows that $\xi:=\mathcal{P}_{F}Dh(x)$ belongs to $\operatorname{int}(C)$. For this $a$ and $\xi$, the minimizer for \eqref{eq:section-support} at $t=0$ is $b=\mathcal{P}_{F}x$, and hence $w=x$. Therefore,
\eq{
 \dot{s}_{\xi,0}(a)=h(x)u(x)\ne 0.
}
Thus $v_{\xi}\not\equiv 0$ and $\mathcal{E}_{\xi}(v_{\xi})>0$. The smooth dependence on $(\xi,a)$ established above shows that the sections and their speeds vary smoothly with $\xi$. Hence $\mathcal{E}_{\xi}(v_{\xi})$ is positive on a neighborhood of this $\xi$. The right-hand side of \eqref{eq:sum-sectional-energy} is therefore positive.
\end{proof}

\begin{proof}[Second proof of \autoref{lem:two-coordinates-reflections}]
Fix $\xi\in\operatorname{int}(C)$, $a\in S^1$, and let $b\in F$ be the minimizer in \eqref{eq:section-support} at $t=0$, and put $w=a+b$. Applying \eqref{eq:for-second-proof} at $t$ and $-t$, we have
\eq{
 s_{\xi,t}(a)&\leq h(w)e^{tu(w)}-\ip{b}{\xi},\\
 s_{\xi,-t}(a)&\leq h(w)e^{-tu(w)}-\ip{b}{\xi},
}
and equality holds at $t=0$. Using $\ip{b}{\xi}\geq 0$, we obtain
\eq{\label{eq:sectional-midpoint}
 s_{\xi,t}(a)s_{\xi,-t}(a)
 &\leq
 \big(h(w)e^{tu(w)}-\ip{b}{\xi}\big)
 \big(h(w)e^{-tu(w)}-\ip{b}{\xi}\big)\\
 &=s_{\xi,0}(a)^2-2h(w)\ip{b}{\xi}\big(\cosh(tu(w))-1\big)\\
 &\leq s_{\xi,0}(a)^2.
}

Recall that $D_{\xi,t}$ and $D_{\xi,-t}$ are origin-symmetric. Let
\eq{
\widetilde{D}_{\xi,t}&=\frac{1}{2}\cdot D_{\xi,t}+_0\frac{1}{2}\cdot D_{\xi,-t}\\
 &:=\left\{y\in E:\ip{a}{y}\leq \sqrt{s_{\xi,t}(a)s_{\xi,-t}(a)} \text{ for every }a\in S^1\right\}.
}
It follows from \eqref{eq:sectional-midpoint} that
\eq{
\widetilde{D}_{\xi,t}\subset D_{\xi,0}.
}

Now the log-Brunn--Minkowski inequality \cite[Thm. 1.3]{BLYZ12} implies that
\eq{
A_{\xi}(0)\geq\operatorname{Area}(\widetilde{D}_{\xi,t})\geq\sqrt{A_{\xi}(t)A_{\xi}(-t)}=A_{\xi}(t),
}
where the last equality follows from $D_{\xi,-t}=R_iD_{\xi,t}$. Hence, by Fubini's theorem,
\eq{
 \operatorname{Vol}(K_t)=\int_C A_{\xi}(t)\,d\xi\leq\int_C A_{\xi}(0)\,d\xi=\operatorname{Vol}(K).
}
That is, $t=0$ is a local maximum of $t\mapsto\operatorname{Vol}(K_t)$. Using \eqref{eq:volume-second-variation}, we conclude that
\eq{
\int_{\Sn}\left(nu^2-|\nabla u|_g^2\right)dV_K=\left.\frac{d^2}{dt^2}\right|_{t=0}\operatorname{Vol}(K_t)\leq 0.
}

It remains to prove strictness of the inequality if $u\not\equiv 0$. Put $\tilde{s}_{\xi,t}=\sqrt{s_{\xi,t}s_{\xi,-t}}$. Note that, for sufficiently small $|t|$, $\tilde{s}_{\xi,t}$ is the support function of $\widetilde{D}_{\xi,t}$. Since $\tilde{s}_{\xi,t}$ is even in $t$, we have $\dot{\tilde{s}}_{\xi,0}=0$. Moreover, $\tilde{s}_{\xi,t}\leq s_{\xi,0}$ and $\tilde{s}_{\xi,0}=s_{\xi,0}$; therefore $\ddot{\tilde{s}}_{\xi,0}\leq 0$ on $S^1$.

Now suppose that $u\not\equiv 0$. Since $n\geq 3$, we may choose $x\in\Sn$ such
that
\eq{
 u(x)\neq 0,\quad \mathcal{P}_Fx\neq 0.
}
Set $a=\mathcal{P}_Ex,\ b=\mathcal{P}_Fx,\ \xi=\mathcal{P}_FDh(x)$. Then $b$ is the minimizer in \eqref{eq:section-support} corresponding to $(\xi,a)$ at $t=0$. Due to \cite[Lem. 2.3]{HI26},
\eq{
 \ip{b}{\xi}
 =\sum_{k\notin\{i,j\}}x_kX_k(x)>0,
}
and in view of \eqref{eq:sectional-midpoint} we have
\eq{
 \tilde{s}_{\xi,0}(a)\ddot{\tilde{s}}_{\xi,0}(a)
 \leq-h(x)\ip{b}{\xi}u(x)^2<0.
}
Hence, using \autoref{lem:area-second-variation}, we obtain
\eq{
 \left.\frac{d^2}{dt^2}\right|_{t=0}
 \operatorname{Area}(\widetilde{D}_{\xi,t})
 =\int_{S^1}\ddot{\tilde{s}}_{\xi,0}\,dS_{D_{\xi,0}}<0.
}

Since $A_{\xi}(t)\leq \operatorname{Area}(\widetilde{D}_{\xi,t})$ and equality holds at $t=0$, it follows that
\eq{
 \ddot{A}_{\xi}(0)
 \leq
 \left.\frac{d^2}{dt^2}\right|_{t=0}
 \operatorname{Area}(\widetilde{D}_{\xi,t})<0.
}
The same estimate holds in a neighborhood of $\xi$. Together with $A_{\eta}(t)\leq A_{\eta}(0)$ for every $\eta\in C$, this yields
\eq{
 \left.\frac{d^2}{dt^2}\right|_{t=0}\operatorname{Vol}(K_t)<0.
}
\end{proof}

\begin{theorem}\label{thm:log-arbitrary}
Let $u\in C^1(\Sn)$. Then \eq{\label{eq:log-arbitrary}
 n\int_{\Omega_+}(u-\bar{u}_+)^2\,dV_K
 \leq \int_{\Omega_+}|\nabla u|_g^2\,dV_K,
 \quad
\bar{u}_+:=\frac{\int_{\Omega_+}u\,dV_K}
 {\int_{\Omega_+}dV_K}.
}
Equality holds only when $u$ is constant on $\Omega_+$.
\end{theorem}
\begin{proof}
This is part of \cite[Thm. 8.12]{HI26}; the unconditionality assumption on $u$ was used there only for the second inequality (8.7).
\end{proof}

\begin{corollary}\label{cor:one-odd-orthogonal}
Let $i\in\{1,\ldots,n\}$. Suppose $u\in C^{\infty}(\Sn)$ is odd with respect to $R_i$, even with respect to every $R_j$, $j\ne i$, and
\eq{\label{eq:li-orthogonal}
 \int_{\Sn}ul_i\,dV_K=0.
}
Then
\eq{
 n\int_{\Sn}u^2\,dV_K
 \leq\int_{\Sn}|\nabla u|_g^2\,dV_K.
}
The inequality is strict unless $u\equiv 0$.
\end{corollary}

\begin{proof}
Let
\eq{
 t=-\frac{\int_{\Omega_+}u\,dV_K}{\int_{\Omega_+}l_i\,dV_K},\quad w=u+tl_i.
}
Then $\int_{\Omega_+}w\,dV_K=0$. Applying \autoref{thm:log-arbitrary} to $w$ yields
\eq{
 n\int_{\Omega_+}w^2\,dV_K
 \leq\int_{\Omega_+}|\nabla w|_g^2\,dV_K.
}
Note that $w^2$ and $|\nabla w|_g^2$ are both unconditional. Therefore
\eq{
 Q(w):=\int_{\Sn}\left(|\nabla w|_g^2-nw^2\right)dV_K\geq 0.
}

On the other hand, by \eqref{eq:li-orthogonal} and $\Delta l_i=-(n-1)l_i$,
\eq{
 \int_{\Sn}g(\nabla u,\nabla l_i)\,dV_K&=(n-1)\int_{\Sn}ul_i\,dV_K=0,\\
 \int_{\Sn}|\nabla l_i|_g^2\,dV_K&=(n-1)\int_{\Sn}l_i^2\,dV_K.
}
Hence
\eq{
 Q(u)=Q(w)+t^2\int_{\Sn}l_i^2\,dV_K\geq 0.
}

If equality holds, then $t=0$. Equality in \autoref{thm:log-arbitrary}, together with $\int_{\Omega_+}u\,dV_K=0$, forces $u=0$ on $\Omega_+$. Thus $u\equiv 0$ on $\Sn$.
\end{proof}

\begin{proof}[Proof of \autoref{thm:main-poincare}]
Following \cite[Proof of Thm. 8.7]{CKLR24}, we write
\eq{
F=\sum_{a\in\{0,1\}^n}F_a,
}
where $F_a$ is even with respect to $R_i$ when $a_i=0$ and odd with respect to $R_i$ when $a_i=1$. For $i=1,\ldots,n$, let $e_i\in\{0,1\}^n$ denote the multi-index whose $i$-th component is $1$ and whose other components are zero. We also write $\abs{a}=a_1+\cdots+a_n$.

Since $K$ is unconditional, every $R_i$ preserves $g$ and $dV_K$. If $a\ne b$, choose $i$ with $a_i\ne b_i$. Then both $F_aF_b$ and $g(\nabla F_a,\nabla F_b)$ are odd with respect to $R_i$, and hence
\eq{
\int_{\Sn}F_aF_b\,dV_K=0,
\quad
\int_{\Sn}g(\nabla F_a,\nabla F_b)\,dV_K=0.
}
Moreover, every $F_a$ with $a\ne 0$ is odd with respect to at least one coordinate reflection and therefore has zero $V_K$-mean.

The component $F_0$ is unconditional, and its mean is zero. Hence, by \cite[Thm. 8.12]{HI26},
\eq{
n\int_{\Sn}F_0^2\,dV_K
\le \int_{\Sn}\abs{\nabla F_0}_g^2\,dV_K,
}
with equality only if $F_0\equiv 0$.

Note that
\eq{
 \int_{\Sn}F_{e_i}l_i\,dV_K=\int_{\Sn}Fl_i\,dV_K=0,
 \quad i=1,\ldots,n.
}
If $\abs{a}=1$, then $a=e_i$ for some $i$. By
\autoref{cor:one-odd-orthogonal},
\eq{
 n\int_{\Sn}F_{e_i}^2\,dV_K
 \leq\int_{\Sn}|\nabla F_{e_i}|_g^2\,dV_K,
}
with equality only if $F_{e_i}\equiv 0$.

Now let $\abs{a}\geq2$. Then $a_i=a_j=1$ for some distinct indices $i$ and $j$, and thus $F_a$ is odd with respect to both $R_i$ and $R_j$. By \autoref{lem:two-coordinates-reflections},
\eq{
n\int_{\Sn}F_a^2\,dV_K \le
\int_{\Sn}\abs{\nabla F_a}_g^2\,dV_K,
}
and the inequality is strict unless $F_a\equiv 0$.

Using the two orthogonal decompositions, we conclude that
\eq{
\int_{\Sn}\abs{\nabla F}_g^2\,dV_K\geq n\int_{\Sn}F^2\,dV_K.
}
If equality holds, the equality cases above force every $F_a$ to vanish.
\end{proof}

\begin{proof}[Proof of \autoref{cor:uniqueness}]
We adapt the technique of \cite{IM23} to the present setting; $X$ is not admissible for applying \autoref{thm:main-poincare}. We write $l=(l_1,\ldots,l_n)$, $V=\frac{1}{n}\int_{\Sn}dV_K$, and
\eq{
A_{ij}=\int_{\Sn}l_i l_j\,dV_K.
}
The matrix $A$ is positive definite. We shall use two identities. First,
\eq{
\int_{\Sn}X_i l_j\,dV_K=V\delta_{ij}.
}
Indeed,
\eq{
\int_{\Sn}X_i\frac{x_j}{h}\,dV_K=\int_{\Sn}X_i x_j\,dS_K=\int_{\partial K}y_i\nu_j(y)\,dy=V\delta_{ij},
}
where $\nu$ denotes the outer unit normal vector of the boundary of $K$.

Second, since $dV_K=h^pdx$ and $p\neq -n$, by \cite[Lem. 4.2]{HI25}, we have
\eq{
\int_{\Sn}x_i x_j\,dV_K=V\delta_{ij}.
}

For $i=1,\ldots,n$, define
\eq{
F_i=X_i-V\sum_{j=1}^n(A^{-1})_{ij}l_j.
}
Then for every $k$,
\eq{
\begin{aligned}
\int_{\Sn}F_il_k\,dV_K
&=\int_{\Sn}X_il_k\,dV_K-V\sum_{j=1}^n(A^{-1})_{ij}\int_{\Sn}l_jl_k\,dV_K\\
&=V\delta_{ik}-V\sum_{j=1}^n(A^{-1})_{ij}A_{jk}=0.
\end{aligned}
}
Moreover, we have
\eq{
\int_{\Sn}F_i\,dV_K=0.
}
Thus, by \autoref{thm:main-poincare}, we have
\eq{\label{ineq:poincare-each-i}
n\int_{\Sn}F_i^2\,dV_K\le \int_{\Sn}\abs{\nabla F_i}_g^2\,dV_K.
}
Summing over $i$ yields
\eq{\label{ineq:poincare-summed}
n\sum_{i=1}^n\int_{\Sn}F_i^2\,dV_K\le 
\sum_{i=1}^n\int_{\Sn}\abs{\nabla F_i}_g^2\,dV_K.
}

Using $\int_{\Sn}X_i l_j\,dV_K=V\delta_{ij}$ and $\int_{\Sn}l_i l_j\,dV_K=A_{ij}$, we get
\eq{
\sum_{i=1}^n\int_{\Sn}F_i^2\,dV_K
=\int_{\Sn}\abs{X}^2\,dV_K-V^2\operatorname{tr}(A^{-1}).
}
Moreover, due to $\Delta l_i=-(n-1)l_i$,
\eq{
\int_{\Sn}g(\nabla l_i,\nabla l_j)\,dV_K=(n-1)A_{ij}
}
and
\eq{
\int_{\Sn}g(\nabla X_i,\nabla l_j)\,dV_K
=-\int_{\Sn}X_i\Delta l_j\,dV_K=(n-1)V\delta_{ij}.
}
Also, we have
\eq{
\sum_{i=1}^n\abs{\nabla X_i}_g^2=h\left(\bar{\Delta} h+(n-1)h\right).
}
Consequently,
\eq{
\sum_{i=1}^n\int_{\Sn}\abs{\nabla F_i}_g^2\,dV_K=\int_{\Sn}h\left(\bar{\Delta} h+(n-1)h\right)\,dV_K-(n-1)V^2\operatorname{tr}(A^{-1}).
}
Substituting these formulas into \eqref{ineq:poincare-summed}, we obtain
\eq{ \label{ineq:key-estimate-L2-norm-X}
n\int_{\Sn}\abs{X}^2\,dV_K\le \int_{\Sn}h\left(\bar{\Delta} h+(n-1)h\right)\,dV_K+V^2\operatorname{tr}(A^{-1}).
}

For every $a\in\R^n$, by the isotropy identity and the Cauchy--Schwarz inequality,
\eq{ 
V\ip{a}{A^{-1}a}
=\int_{\Sn}\ip{hx}{a}\ip{l}{A^{-1}a}\,dV_K
\le
\left(\int_{\Sn}h^2\ip{x}{a}^2\,dV_K\right)^{1/2}\ip{a}{A^{-1}a}^{1/2}.
}
Here we used 
\eq{
\int_{\Sn}\ip{l}{A^{-1}a}^2\,dV_K
=\ip{A^{-1}a}{A(A^{-1}a)}
=\ip{a}{A^{-1}a}.
}
Hence $V^2\ip{a}{A^{-1}a}\le \int_{\Sn}h^2\ip{x}{a}^2\,dV_K$. Taking traces yields
\eq{
V^2\operatorname{tr}(A^{-1})\le \int_{\Sn}h^2\,dV_K,
}
and thus
\eq{ \label{ineq:poincare-for-|X|^2}
n\int_{\Sn}\abs{X}^2\,dV_K\le \int_{\Sn}h\left(\bar{\Delta} h+(n-1)h\right)\,dV_K+\int_{\Sn}h^2\,dV_K.
}

Now using $dV_K=h^pdx$ and $\abs{X}^2=h^2+\abs{\bar{\nabla} h}_{\bar{g}}^2$, we arrive at
\eq{\label{ineq:poincare-for-h}
n\int_{\Sn}h^{p+2}\,dx+n\int_{\Sn}h^p\abs{\bar{\nabla} h}_{\bar{g}}^2\,dx\le \int_{\Sn}h^{p+1}\bar{\Delta} h\,dx+n\int_{\Sn}h^{p+2}\,dx.
}
Since
\eq{
\int_{\Sn}h^{p+1}\bar{\Delta} h\,dx=-(p+1)\int_{\Sn}h^p\abs{\bar{\nabla} h}_{\bar{g}}^2\,dx,
}
we obtain
\eq{
(n+p+1)\int_{\Sn}h^p\abs{\bar{\nabla} h}_{\bar{g}}^2\,dx\le0.
}
Hence, when $p>-n-1$, we conclude that $h=1$.

It remains to consider $p=-n-1$. In this case, equality holds in  \eqref{ineq:poincare-for-h}; consequently, equality also holds in \eqref{ineq:poincare-summed} and hence for each $F_i$:
\eq{
n\int_{\Sn}F_i^2\,dV_K
=\int_{\Sn}\abs{\nabla F_i}_g^2\,dV_K,\quad i=1,\ldots,n.
}
By the equality characterization in \autoref{thm:main-poincare},
\eq{
F_i\equiv 0,\quad i=1,\ldots,n.
}
Thus,
\eq{
X=M\frac{x}{h},\quad \text{where}\quad M:=VA^{-1}\implies h^2=x^TMx.
}
Thus $K$ is a centered ellipsoid, and
\eq{
h^{-n-2}=\det\left(\bar{\nabla}^2h+h\bar{g}\right)=\det(M)\,h^{-(n+1)}.
}
It follows that $h=1$.
\end{proof}

\section*{Acknowledgment}
Hu was supported by the National Key Research and Development Program of China (Grant No. 2021YFA1001800). Ivaki was supported by the Austrian Science Fund (FWF) under Project P36545.

\vspace{5mm}

\noindent
\begin{minipage}[t]{0.40\textwidth}
\textsc{School of Mathematical\\ Sciences, Beihang University,\\ Beijing 100191, China}\\
\email{\href{mailto:huyingxiang@buaa.edu.cn}{huyingxiang@buaa.edu.cn}}
\end{minipage}
\hfill
\begin{minipage}[t]{0.40\textwidth}
\textsc{Institut f\"{u}r Diskrete\\ Mathematik und Geometrie,\\ Technische Universit\"{a}t Wien, Wiedner Hauptstra{\ss}e 8--10,\\ 1040 Wien, Austria}\\
\email{\href{mailto:mohammad.ivaki@tuwien.ac.at}{mohammad.ivaki@tuwien.ac.at}}
\end{minipage}

\end{document}